\documentclass[12pt]{amsart}
\usepackage{amsfonts}
\usepackage{amsfonts,latexsym,rawfonts,amsmath,amssymb,amsthm}
\usepackage[plainpages=false]{hyperref}
\usepackage{graphicx}

\numberwithin{equation}{section}

\RequirePackage{color}
\theoremstyle{plain}

\newtheorem{theorem}{Theorem}[section]

\newtheorem{lemma}[theorem]{Lemma}
\newtheorem{proposition}[theorem]{Proposition}
\newtheorem{remark}[theorem]{Remark}

\newcommand{\beq}{\begin{equation}}
\newcommand{\eeq}{\end{equation}}
\newcommand{\beqs}{\begin{eqnarray*}}
\newcommand{\eeqs}{\end{eqnarray*}}
\newcommand{\beqn}{\begin{eqnarray}}
\newcommand{\eeqn}{\end{eqnarray}}
\newcommand{\beqa}{\begin{array}}
\newcommand{\eeqa}{\end{array}}

\def\phi{\varphi}

\begin{document}
\title[Prescribed Weingarten curvatures]{A class of  Hessian quotient equations  in the
warped product manifold}

\author{Xiaojuan Chen}
\address{Faculty of Mathematics and Statistics, Hubei Key Laboratory of Applied Mathematics, Hubei University,  Wuhan 430062, P.R. China}
\email{201911110410741@stu.hubu.edu.cn}

\author{Qiang Tu$^{\ast}$}
\address{Faculty of Mathematics and Statistics, Hubei Key Laboratory of Applied Mathematics, Hubei University,  Wuhan 430062, P.R. China}
\email{qiangtu@hubu.edu.cn}

\author{Ni Xiang}
\address{Faculty of Mathematics and Statistics, Hubei Key Laboratory of Applied Mathematics, Hubei University,  Wuhan 430062, P.R. China}
\email{nixiang@hubu.edu.cn}



\thanks{This research was supported by funds from Hubei Provincial Department of Education
Key Projects D20181003, Natural Science Foundation of Hubei Province, China, No. 2020CFB246 and the National Natural Science Foundation of China No. 11971157.}
\thanks{$\ast$ Corresponding author}

\begin{abstract}
In this paper, we consider a class of Hessian quotient equations in the warped product manifold
$\overline{M}=I\times_{\lambda}M$. Under some sufficient conditions, we obtain an existence result for the star-shaped compact hypersurface $\Sigma$ in $\overline{M}$ using standard degree theory based on a priori estimates for solutions to the Hessian quotient equations.
\end{abstract}

\maketitle

\baselineskip18pt

\parskip3pt

 \section{Introduction}

In this paper, we consider the problem of prescribed Weingarten
curvatures for closed, star-shaped hypersurfaces in the warped product
manifold.
 Let $(M,g')$ be a compact Riemannian manifold and $I$ be an open interval in $\mathbb{R}$. The warped product manifold $\overline{M}=I\times_{\lambda}M$ is endowed with the metric
\begin{eqnarray}\label{metric}
\overline{g}^2=dr^2+\lambda^2(r) g',
\end{eqnarray}
where $\lambda:I\rightarrow\mathbb{R}^{+}$ is a positive $C^2$ differentiable function.
Let $\Sigma$ be a  compact star-shaped hypersurface in $\overline{M}$, thus $\Sigma$ can be parametrized as a radial graph over $M$. Specifically speaking,
there exists a differentiable function
$r : M \rightarrow I$ such that the  graph of $\Sigma$ can be represented by
\begin{equation*}
  \Sigma=\{(r(u),u)\mid u\in{M}\}.
\end{equation*}
We consider the following prescribed Weingarten curvature equation
\begin{eqnarray}\label{Eq}
\frac{\sigma_{k}}{\sigma_l}(\mu(\eta))=f(V, \nu(V)), \quad \forall~V  \in \overline{M},
\end{eqnarray}
 where $2\leq k\leq n, 0\leq l \leq k-2$, $V=\lambda\frac{\partial}{\partial r}$ is the position vector field of hypersurface $\Sigma$ in $\overline{M}$, $\sigma_k$ is the $k$-th elementary symmetric function, $\mu(\eta)$ is the eigenvalue of $g^{-1}\eta$, $f$ is a given smooth function and  $\nu(V)$ is the unit outer normal vector  at $V$. The   $(0,2)$-tensor $\eta$  on $\Sigma$ is defined by
$$\eta_{ij}= Hg_{ij}- h_{ij},$$
where $g_{ij}$ and $h_{ij}$ are the first and second fundamental forms of $\Sigma$ respectively, $H(V)$ is the mean curvature at $V \in \Sigma$. In fact, $\eta$ is the first Newton transformation  of $h$ with respect to $g$.
Given $r_1$, $r_2$ with $r_1<r_2$, we define the annulus domain $\{(r,u)\in \overline{M} \mid r_1\leq r \leq r_2\}$.
The main theorem is as follows.
\begin{theorem}\label{Main}
Let $M$ be a compact Riemannian manifold, $\overline{M}$ be the warped product manifold
with the metric (\ref{metric}) and $\Gamma$ be an open neighborhood of unit normal bundle of $M$ in $\overline{M}\times \mathbb{S}^n$.  Assume that $\lambda$ is a positive $C^2$ differentiable function and $\lambda'>0$. Suppose that $f$ satisfies\par
\begin{eqnarray}\label{ASS1}
f(V,\nu)>\frac{C_n^k}{C_n^l}((n-1)\zeta(r))^{k-l},\quad \quad \forall~ r\leq r_{1},
\end{eqnarray}
\begin{eqnarray}\label{ASS2}
f(V,\nu)<\frac{C_n^k}{C_n^l}((n-1)\zeta(r))^{k-l},\quad \quad \forall ~r \geq r_{2}
\end{eqnarray}
and
\begin{eqnarray}\label{ASS3}
\frac{\partial}{\partial r}(\lambda^{k-l}f(V,\nu))\leq 0, \quad \quad \forall~ r_{1}<r<r_{2},
\end{eqnarray}
where $V=\lambda\frac{\partial}{\partial r}$ and $\zeta(r)=\lambda'(r)/\lambda(r)$.
Then there exists a $C^{4, \alpha}$, $(\eta,k)$-convex, star-shaped and closed hypersurface $\Sigma$ in $\{(r,u)\in \overline{M}\mid r_{1}\leq r \leq r_{2}\}$ satisfies the equation \eqref{Eq} for any $\alpha\in (0,1)$.
\end{theorem}

\begin{remark}
The key to prove Theorem \ref{Main} is to obtain the curvature estimate for
the Hessian quotient equation (\ref{Eq}) in the warped product manifold, which is established in Theorem \ref{n-2-C2e}.
\end{remark}

This kind of Hessian quotient equation is stimulated by many important geometric
problems. 
When $k=n$, $l=0$ and $\lambda(r)=r$, the equation \eqref{Eq} becomes the following equation for an $(\eta, n)$-convex hypersurface
\begin{eqnarray}\label{ht-Eq-8}
\mbox{det} (\eta(V))=f(V, \nu),
\end{eqnarray}
which is studied intensively by Sha \cite{Sha1, Sha2}, Wu \cite{Wu} and Harvey-Lawson \cite{HL2}.
When the left hand of \eqref{ht-Eq-8} is replaced by $\sigma_k(\eta(V))$, Chu-Jiao
established curvature estimates for this kind of equation in \cite{CJ}.
Inspired by above results, the authors in \cite{CTX} considered the corresponding Hessian quotient type
prescribed curvature equations in Euclidean space. In this paper, we generalize the existence results in \cite{CTX} to the warped product manifold for the prescribed curvature problem. The remarkable fact is that Theorem \ref{Main} recovers the existence results in \cite{CJ, CTX}.

 When $\mu(\eta)$ is replaced  by $\kappa(V)$ and $l=0$, the equation \eqref{Eq} becomes this kind of  prescribed curvature equation
\begin{eqnarray}\label{ht-Eq-9}
\sigma_k(\kappa(V))=f(V, \nu),
\end{eqnarray}
which has been widely studied in the past two decades.
The key to this prescribed curvature equation is the curvature estimate.
 In Euclidean space, Caffarelli-Nirenberg-Spruck established the curvature estimates for $k=n$ in \cite{Ca1}.
Guan-Ren-Wang proved the $C^2$ estimates of the equation  \eqref{ht-Eq-9} for $k=2$ in \cite{Guan-Ren15}.  Spruck-Xiao extended the $2$-convex case to space forms and gave a simple proof for the Euclidean case in \cite{Sp}. Ren-Wang proved the $C^2$ estimates for $k=n-1$ and $n-2$ in \cite{Ren, Ren1}. When $2<k<n$, the $C^2$ estimates for the equation of prescribing curvature
measures were also proved in \cite{Guan12, Guan09}, where $f(V, \nu) = \langle V, \nu \rangle \Tilde{f}(V)$.
Ivochkina considered the Dirichlet problem of the equation \eqref{ht-Eq-9} on domains in $\mathbb{R}^n$ and obtained the $C^2$ estimates according to the dependence of $f$ on $\nu$ under some extra conditions in \cite{Iv1,Iv2}.
 Caffarelli-Nirenberg-Spruck \cite{Ca} and Guan-Guan \cite{Guan02} proved the $C^2$ estimates when $f$ was independent of $\nu$ and depended only on $\nu$ respectively. Moreover, some results have been derived by Li-Oliker \cite{Li-Ol} on unit sphere, Barbosa-de Lira-Oliker \cite{Ba-Li} on space forms, Jin-Li \cite{Jin} on hyperbolic space and  Andrade-Barbosa-de Lira \cite{Al} on the warped product manifold. In particular, Chen-Li-Wang \cite{CLW} generalized the results in \cite{Guan-Ren15} and Ren-Wang \cite{Ren} extended to the $(n-2)$-convex hypersurface in the warped product manifold.

The organization of the paper is as follows.
In Sect. 2
we start with some preliminaries.
The $C^0$, $C^1$ and $C^2$ estimates are given in Sect. 3.
In Sect. 4 we finish the proof of Theorem \ref{Main}.

\section{Preliminaries}

\subsection{Star-shaped hypersurfaces in the warped product
manifold}

Let $M$ be a compact Riemannian manifold with the metric $g'$ and $I$ be an open interval in $\mathbb{R}$. Assume that $\lambda : I\rightarrow \mathbb{R}^{+}$ is a positive differential function and
$\lambda'>0$. Clearly,
\begin{eqnarray*}
\lambda(r)=\begin{cases}r \quad \quad \mbox{on}~[0,\infty)\\
\sin r\quad \mbox{on}~[0,\frac{\pi}{2})\\
\sinh r \quad \mbox{on}~ [0,\infty)
\end{cases} \Longrightarrow \overline{M}= \begin{cases}\mathbb{R}^{n+1}\\
\mathbb{S}^{n+1}\\
\mathbb{H}^{n+1}.
\end{cases}
\end{eqnarray*}
 The manifold $\overline{M}=I\times_{\lambda}M$ is called the warped product if it is endowed with the metric
\begin{eqnarray*}
\overline{g}^{2}=dr^2+\lambda^2(r) g'.
\end{eqnarray*}
The metric in $\overline{M}$ is denoted by $\langle\cdot,\cdot\rangle$. The corresponding Riemannian connection in $\overline{M}$ will be denoted by $\overline{\nabla}$. The usual connection in $M$ will be denoted by $\nabla'$. The curvature tensors in $M$ and $\overline{M}$ will be denoted by $R$ and $\overline{R}$ respectively.

Let $\{e_1,\cdots,e_{n-1}\}$ be an orthonormal frame field in M and let $\{\theta_1,\cdots,\theta_{n-1}\}$ be the associated dual frame.
The connection forms $\theta_{ij}$ and curvature forms $\Theta_{ij}$ in M satisfy the structural equations
  \begin{align*}
    &d\theta_i=\sum_j\theta_{ij}\wedge\theta_j,\quad \theta_{ij}=-\theta_{ji} ,\\
    &d\theta_{ij}-\sum_k\theta_{ik}\wedge\theta_{kj}=\Theta_{ij}=-\frac{1}{2}\sum_{k,l}R_{ijkl}\theta_k
    \wedge\theta_l.
  \end{align*}
An orthonormal frame in $\overline{M}$ may be defined by $\overline{e}_i=\frac{1}{\lambda}e_i,1\leq i\leq n-1$ and $\overline{e}_0=\frac{\partial}{\partial r}$. The associated dual frame is that $\overline{\theta}_i=\lambda\theta_i$ , $1\leq i \leq n-1$ and $\overline{\theta}_0=dr$.
Then, we have the following lemma (See \cite{HLW}).

\begin{lemma}
Given a differentiable function $r : M \rightarrow I$, its graph is defined by the hypersurface
\begin{eqnarray*}
\Sigma=\{(r(u),u): u \in M\}.
\end{eqnarray*}
Then, the tangential vector takes the
form
$$V_i=\lambda\overline{e}_i+r_i \overline{e}_0,$$
where $r_i$ are the components of the differential $dr=r_i \theta^i$.
The induced metric on $\Sigma$ has
$$g_{ij}=\lambda^2(r)\delta_{ij}+r_i r_j,$$
and its inverse is given by
$$g^{ij}=\frac{1}{\lambda^2}(\delta_{ij}-\frac{r^i r^j}{v^2}).$$
We also have the outward unit normal vector of $\Sigma$
$$\nu=\frac{1}{v}\bigg(\lambda \overline{e}_0 -r^i\overline{e}_i\bigg),$$
where $v=\sqrt{\lambda^2+|\nabla'r|^2}$ with $\nabla'r=r^ie_i$.
Let $h_{ij}$ be the second fundamental form of $\Sigma$ in term of the tangential vector fields $\{X_1,
..., X_n\}$. Then,
$$h_{ij}=-\langle\overline{\nabla}_{X_j} X_i, \nu\rangle=
\frac{1}{v}\bigg(-\lambda r_{ij}+2\lambda'r_ir_j+\lambda^2\lambda' \delta_{ij}\bigg)$$
and
$$h^i_j=\frac{1}{\lambda^2 v}(\delta_{ik}-\frac{r^i r^k}{v^2})\bigg(-\lambda r_{kj}+2\lambda'r_kr_j+\lambda^2\lambda' \delta_{kj}\bigg),$$
where $r_{ij}$ are the components of the Hessian $\nabla'^{2} r=\nabla' dr $ of $r$ in $M$.
\end{lemma}

Let $\Gamma_k$ be the connected component of $\{\kappa\in \mathbb{R}^n \mid \sigma_m>0, m=1,\cdots, k\}$, the operator $\sigma_k(\kappa)$ for $\kappa=(\kappa_1, \cdots, \kappa_n)\in \Gamma_k$ is defined by
$$\sigma_k(\kappa)=\sum_{1\leq i_1<\cdots<i_k\leq n} \kappa_{i_1} \kappa_{i_2}\cdots \kappa_{i_k}.$$
A smooth hypersurface $M \subset \mathbb{R}^{n+1}$ is called
$(\eta, k)$-convex if $\mu(\eta) \in \Gamma_k$ for any $V\in M$,
 where $\Gamma_k$ is the Garding cone
\begin{eqnarray*}\label{cone}
\Gamma_{k}=\{\lambda \in \mathbb{R} ^n: \sigma_{j}(\mu)>0, \forall ~ 1\leq j \leq k\}.
\end{eqnarray*}

For convenience, we introduce the following notations:
\begin{eqnarray*}
G(\eta):= \left(\frac{\sigma_k(\eta)}{\sigma_l(\eta)}\right)^{\frac{1}{k-l}},\quad G^{ij}:=\frac{\partial G}{ \partial \eta_{ij}}, \quad
G^{ij, rs}:= \frac{\partial^2 G}{\partial \eta_{ij} \partial \eta_{rs}}, \quad
 F^{ii}:=\sum_{k\neq i} G^{kk}.
\end{eqnarray*}
Thus,
$$G^{ii}= \frac{1}{k-l} \left(\frac{\sigma_k(\eta)}{\sigma_l(\eta)}\right)^{\frac{1}{k-l}-1} \frac{\sigma_{k-1}(\eta| i)\sigma_l(\eta)-\sigma_k(\eta)\sigma_{l-1}(\eta| i)}{\sigma_l^2(\eta)}.$$
If $\eta=\mbox{diag}(\mu_1, \mu_2, \cdots, \mu_n)$ and $\mu_1 \leq \mu_2\leq \cdots \leq \mu_n$, then we have
$$G^{11}\geq G^{22} \geq \cdots \geq G^{nn}, \quad F^{11} \leq F^{22} \leq \cdots \leq F^{nn}.$$
Note that $\sum_iF^{ii} = (n-1) \sum_iG^{ii} \geq (n-1) \left(\frac{C_n^k}{C_n^l}\right)^{\frac{1}{k-l}}$ and
\begin{equation}\label{Fi}
  F^{ii} \geq F^{22} \geq \frac{1}{n(n-1)} \sum_i F^{ii}.
\end{equation}

To handle the ellipticity of the equation \eqref{Eq}, we need the
following important propositions and their proof are the same as
Proposition 2.2.3 in \cite{CQ1}.
\begin{proposition}\label{th-lem-07}
%Let $\lambda=(\lambda_1, \cdots, \lambda_n) \in \Gamma_k$  be a diagonal
Let $\eta$ be a diagonal matrix with $\mu(\eta)\in \Gamma_k$, $0\leq l \leq k-2$ and $k\geq 3$. Then
\begin{eqnarray*}
-G^{1i, i1}(\eta)=\frac{G^{11}-G^{ii}}{\eta_{ii}-\eta_{11}}, \quad \forall~ i\geq 2.
\end{eqnarray*}
\end{proposition}

\begin{proposition}\label{ellipticconcave}
Let $M$ be a smooth $(\eta, k)$-convex closed hypersurface in $\mathbb{R}^{n+1}$
and $0\leq l< k-1$. Then the operator
\begin{eqnarray*}
G(\eta_{ij}(V))=\left(\frac{\sigma_k(\mu(\eta))}{\sigma_{l}(\mu(\eta))}\right)^{\frac{1}{k-l}}
\end{eqnarray*}
is elliptic and concave with respect to $\eta_{ij}(V)$. Moreover we have
\begin{eqnarray*}
\sum G^{ii} \geq \left(\frac{C_n^k}{C_n^l}\right)^{\frac{1}{k-l}}.
\end{eqnarray*}
\end{proposition}

%%%%%%%%%%%%%%%%%%%%%%
%% 从这里开始打       %%
%Now we choose the coordinate systems such that $\{E_{0}=\nu,E_1,\cdots,E_n\}$ is an orthonormal frame field in some open sets of $\Sigma$ and $\{\omega_0,\omega_1,\cdots,\omega_n\}$ is  its associated dual frame. The connection forms $\{\omega_{ij}\}$ and curvature forms $\{\Omega_{ij}\}$ in $\Sigma$  satisfy the structural equations
%\begin{equation*}
%\begin{split}
%   & d\omega_i-\sum_j\omega_{ij}\wedge\omega_j=0,\quad \omega_{ij}+\omega_{ji}=0,\\ & d\omega_{ij}-\sum_k\omega_{ik}\wedge\omega_{kj}=\Omega_{ij}=-\frac{1}{2}\sum_{kl}R_{ijkl}\omega_k  \wedge\omega_l.
%\end{split}
%\end{equation*}
%The coefficients $h_{ij}$ of the second fundamental form are given by Weingarten equation
%\begin{equation*}
 % \omega_{i0}=\Sigma_jh_{ij}\omega_j.
%\end{equation*}
%The covariant derivative of the second fundamental form $h_{ij}$ in $\Sigma$ is defined by
%\begin{equation*}
%\begin{split}
 %  & \sum_kh_{ijk}\omega_k = dh_{ij}+\sum_lh_{il}\omega_{lj}+\sum_lh_{lj}\omega_{li},\\
  %  & \sum_lh_{ijkl}\omega_l = dh_{ijk}+\sum_lh_{ljk}\omega_{li}+\sum_lh_{ilk}\omega_{lj}+\sum_lh_{ijl}
  %\omega_{lk}.
%\end{split}
%\end{equation*}
The Codazzi equation is a commutation formula for the first order derivative of $h_{ij}$ given by
\begin{equation*}
  h_{ijk}-h_{ikj}=\overline{R}_{0ijk}
\end{equation*}
and the Ricci identity is a commutation formula for the second order derivative of $h_{ij}$ given by \cite[Lemma 2.2]{CLW}, the following lemma can be derived.
\begin{lemma}
   Let $\overline{X}$ be a point of $\Sigma$ and $\{E_{0}=\nu,E_1,\cdots,E_n\}$ be an adapted frame field such that each $E_i$ is a principal direction and  $\omega_i^k=0$ at $\overline{X}$. Let $(h_{ij})$ be the second quadratic form of $\Sigma$. Then, at the point $\overline{X}$, we have
   %\begin{equation}
    % \begin{split}
     %  h_{llii}= &h_{iill}-h_{lm}(h_{mi}h_{il}-h_{ml}h_{ii})-h_{mi}(h_{mi}h_{ll}-h_{ml}h_{li}) \\
      % &+\overline{R}_{0iil;l}-2h_{ml}\overline{R}_{miil}+h_{il}\overline{R}_{0i0l}+h_{ll}\overline{R}
       %_{0ii0}\\
       %&+\overline{R}_{0lil;i}-2h_{mi}\overline{R}_{mlil}+h_{ii}\overline{R}_{0l0l}+h_{li}\overline{R}_{0li0}.
     %\end{split}
   %\end{equation}
   %In particular, we have
   \begin{equation}\label{hii}
     h_{ii11}-h_{11ii}=h_{11}h_{ii}^2-h_{11}^2h_{ii}+2(h_{ii}-h_{11})\overline{R}_{i1i1}+h_{11}\overline{R}
     _{i0i0}-h_{ii}\overline{R}_{1010}+\overline{R}_{i1i0;1}-\overline{R}_{1i10;i}.
   \end{equation}
\end{lemma}

%%%%%%%%%%%%%%%%%%%%%%
Consider the function
\begin{eqnarray*}
\tau=\langle V, \nu\rangle, \quad \quad  \Lambda(r)=\int_{0}^{r}\lambda(s) d s
\end{eqnarray*}
with the position vector field
\begin{eqnarray*}
V=\lambda(r)\frac{\partial}{\partial_r}.
\end{eqnarray*}
 Then, we need the following lemma for $\tau$ and  $\Lambda$.
\begin{lemma}\label{supp}
Let $\tau$, $\Lambda$ be functions as above, then we have
\begin{eqnarray}\label{1d-lad}
\nabla_{E_i} \Lambda =\lambda \langle \overline{e}_0, E_i\rangle E_i,
\end{eqnarray}
\begin{eqnarray}\label{1d-tau}
\nabla_{E_i} \tau = \sum_j \left(\nabla_{E_j}\Lambda\right) h_{ij},
\end{eqnarray}
\begin{eqnarray}\label{2d-lad}
\nabla^2_{E_i, E_j} \Lambda=\lambda^{\prime}g_{ij}-\tau h_{ij}
\end{eqnarray}
and
\begin{eqnarray}\label{2d-tau}
\nabla^2_{E_i, E_j} \tau=-\tau\sum_k h_{ik}h_{kj}+ \lambda^{\prime}h_{ij}+\sum_k \left( h_{ijk}-\overline{R}_{0ijk}\right) \nabla_{E_k} \Lambda.
\end{eqnarray}
\end{lemma}
\begin{proof}
See Lemma 2.2, Lemma 2.6 and Lemma 2.3 in \cite{CLW}, \cite{Guan15} or \cite{Jin} for the details.
\end{proof}

%\begin{lemma}\label{n-2-pre-lem1}
%If $\kappa\in\Gamma_k$ and $\kappa_1\geq\cdots\geq\kappa_k\geq\cdots\geq\kappa_n$, then we have

%(a)
%\[ \sigma_{k-1}(\kappa)\geq\theta(n,k)\kappa_1\kappa_2\cdots\kappa_{k-1},\]where $\theta(n,k)>0$,

%(b)
%\[\sigma_k(\kappa)\leq C_n^k\kappa_1\cdots\kappa_k,\]

%(c)
%\[\sigma_{k-1}(\kappa|k)\geq C(n,k)\sigma_{k-1}(\kappa),
%\]

%(d)
%\[-\kappa_i<\frac{(n-k)\kappa_1}{k},\]
%if $\kappa_i\leq0$, $1\leq i\leq n$.
%\end{lemma}
%\begin{proof}
%See Proposition 1.2.7, 1.2.9, Corollary 1.2.11 in\cite{CQ1} and Lemma 8 in \cite{Ren2} for the proof.
%\end{proof}

%%%%%%%%%%%%%%%%%%%%%%%%%%%%%%%%%%%%%%%%%%%%%%%%%%%%%%%%%%%%%%%%%%%%%%%%%%%%%%
%%%%%%%%%%%%%%%%%%%%%%%%%%%%%%%%%%%%%%%%%%%%%%%%%%%%%%%%%%%%%%%%%%%%%%%%%%%%%%
\section{A  priori estimates}

In order to prove Theorem \ref{Main}, we use the degree theory for the
nonlinear elliptic equation developed in \cite{Li89} and the proof
here is similar to those in \cite{Al,Jin,Li-Sh,Li-Ol}. First, we consider
the family of equations for $0\leq t\leq 1$
\begin{eqnarray}\label{Eq2}
\frac{\sigma_k(\mu(\eta))}{\sigma_l(\mu(\eta))}=f^t( V, \nu(V)),
\end{eqnarray}
where
$f^t=tf(r,u, \nu)+(1-t)\phi(r) \frac{C_n^k}{C_n^l}((n-1)\zeta(r))^{k-l}$, $\zeta(r)=\frac{\lambda'}{\lambda}$ and $\phi$ is a positive function which satisfies the following conditions:

(a) $\phi(r)>0$,

(b) $\phi(r)\geq 1$ for $r\leq r_1$,

(c) $\phi(r)\leq 1$ for $r\geq r_2$,

(d) $\phi^{\prime}(r)<0$.

%%%%%%%%%%%%%%%%%%%%%%%%%%%%%%%%%%%%%%%%%%%%%%%%%%%%%%%%%%%%%%%%%%%%%%%%%%%%%%
\subsection{$C^0$ Estimates}

Now, we can prove the following proposition which asserts that the
solution of the equation \eqref{Eq} has uniform $C^0$ bounds.

\begin{proposition}\label{n-2-C^0}
Under the assumptions \eqref{ASS1} and \eqref{ASS2}, if the $(\eta,k)$-convex hypersurface ${\Sigma}=\{(r(u),u)\mid
u \in M \}\subset \overline{M}$ satisfies the equation
\eqref{Eq2} for a given $t \in (0, 1]$, then
\begin{eqnarray*}
r_1<r(u)<r_2, \quad \forall \ u \in M.
\end{eqnarray*}
\end{proposition}

\begin{proof}
Assume $r(u)$ attains its maximum at $u_0 \in M$ and
$r(u_0)\geq r_2$, then recall
\begin{eqnarray*}
h^i_j=\frac{1}{\lambda^2 v}(\delta_{ik}-\frac{r^i r^k}{v^2})\bigg(-\lambda r_{kj}+2\lambda'r_kr_j+\lambda^2\lambda' \delta_{kj}\bigg),
\end{eqnarray*}
which implies together with the fact that the matrix $r_{ij}$ is
non-positive definite at $u_0$
\begin{eqnarray*}
h^{i}_{j}(u_0)=\frac{1}{\lambda^3}\bigg(-\lambda r_{ij}+\lambda^2\lambda' \delta_{ij}\bigg)\geq \frac{\lambda'}{\lambda} \delta_{ij}.
\end{eqnarray*}
Then
\begin{equation*}
  \eta^i_j(u_0)= H \delta^i_j-h^i_j\geq\frac{(n-1)\lambda'}{\lambda}\delta_{ij}.
\end{equation*}
Note that $\frac{\sigma_{k}}{\sigma_{l}}$ for $0\leq l\leq k-2$ is concave in
$\Gamma_{k}$. Thus
\begin{eqnarray*}
\frac{\sigma_k(\mu(\eta))}{\sigma_{l}(\mu(\eta))} \geq
\frac{\sigma_k(\frac{(n-1)\lambda'}{\lambda}\delta_{ij})}{\sigma_{l}(\frac{( n-1)\lambda'}{\lambda}\delta_{ij})}=\frac{C_n^k}{C_n^l}(\frac{( n-1)\lambda'}{\lambda})^{k-l}=\frac{C_n^k}{C_n^l}((n-1)\zeta(r))^{k-l}.
\end{eqnarray*}
So, we arrive at $u_0$
\begin{eqnarray*}
tf(r,u, \nu)+(1-t)\phi(r) \frac{C_n^k}{C_n^l}((n-1)\zeta(r))^{k-l}\geq
\frac{C_n^k}{C_n^l}((n-1)\zeta(r))^{k-l}.
\end{eqnarray*}
Thus, we obtain at $u_0$
\begin{eqnarray*}
f(r,u, \nu)\geq \frac{C_n^k}{C_n^l}((n-1)\zeta(r))^{k-l},
\end{eqnarray*}
which is in contradiction to \eqref{ASS2}. Thus, we have $r(u)<
r_2$ for $u \in M$. Similarly, we can obtain $r(u)> r_1$
for $u \in M$.
\end{proof}

Now, we prove the following uniqueness result.

\begin{proposition}\label{Uni}
There exists an unique $(\eta,k)$-convex solution to the
equation \eqref{Eq2} with $t=0$, namely $\Sigma_0=\{(r(u),  u) \in \overline{M} \mid
r(u)=r_0\}$, where $r_0$ satisfies $\varphi(r_0)=1$.
\end{proposition}

\begin{proof}
Let $\Sigma_0$ be a solution of \eqref{Eq2} for $t=0$, then
\begin{eqnarray*}
\frac{\sigma_k}{\sigma_l}(\mu(\eta))-\phi(r) \frac{C_n^k}{C_n^l}((n-1)\zeta(r))^{k-l}=0.
\end{eqnarray*}
Assume $r(u)$ attains its maximum $r_{max}$ at $u_0 \in
M$, then we have at $u_0$
\begin{equation*}
h^{i}_{j}=\frac{1}{\lambda^3}\bigg(-\lambda r_{ij}+\lambda^2\lambda' \delta_{ij}\bigg),
\end{equation*}
which implies together with the fact that the matrix $r_{ij}$ is
non-positive definite at $u_0$
\begin{eqnarray*}
\frac{\sigma_k(\mu(\eta))}{\sigma_{l}(\mu(\eta))} \geq
\frac{C_n^k}{C_n^l}((n-1)\zeta(r))^{k-l}.
\end{eqnarray*}
By the equation \eqref{Eq2}
\begin{eqnarray*}
\varphi(r_{max})\geq 1.
\end{eqnarray*}
Similarly,
\begin{eqnarray*}
\varphi(r_{min})\leq 1.
\end{eqnarray*}
Thus, since $\varphi$ is a decreasing function, we obtain
\begin{eqnarray*}
\varphi(r_{min})=\varphi(r_{max})=1.
\end{eqnarray*}
We conclude
\begin{equation*}
  r(x)=r_0, \quad\forall~( r(u), u) \in \overline{M},
\end{equation*}
where $r_0$ is the unique solution of
$\varphi(r_0)=1$.
\end{proof}

%%%%%%%%%%%%%%%%%%%%%%%%%%%%%%%%%%%%%%%%%%%%%%%%%%%%%%%%%%%%%%%%%%%%%%%%%%%%%%
\subsection{$C^1$ Estimates}
%%%%%%%%%%%%%%%%%%%%%%%%%%%%%%%%%%%%%%%%%%%%%%%%%%%%%%%%%%%%%%%%%%%%%%%%%%%%%%
In this section, we establish the gradient estimates for the equation \eqref{Eq2}.
The treatment of this section follows from \cite[Lemma 3.1]{CLW}.

We recall that a star-shaped hypersurface $\Sigma$ in $\overline{M}$ can be represented by
\begin{equation*}
  \Sigma=\{V(u)=(r(u),u)\mid u\in{M}\},
\end{equation*}
where $V$ is the position vector field of hypersurface $\Sigma$ in $\overline{M}$.
We define a function $\tau=\langle V,\nu\rangle$. It is clear that
$$\tau=\frac{r^2}{\sqrt{r^2+|Dr|^2}}.$$
\begin{theorem}\label{n-2-C1e}
Under the assumption \eqref{ASS3}, if the closed star-shaped  $(\eta,k)$-convex hypersurface $\Sigma=\{(r(u),u)\in \overline{M}\mid u \in M\}$ satisfies the curvature equation \eqref{Eq2}
 and $r$ has positive upper and lower
bound, then there exists a constant C depending only on $n, k, l, \|\lambda\|_{C^1}, \inf_{\Sigma} r,  \sup_{\Sigma} r, \inf_{\Sigma}f$ and $\|f\|_{C^1}$ such that
\begin{equation*}
|D r|\leq C.
\end{equation*}
\end{theorem}

\begin{proof}
It is sufficient to obtain a positive lower bound of $\tau$. We consider the function
\begin{equation*}
  \Phi=-\ln\tau+\gamma(\Lambda),
\end{equation*}
where $\gamma(\Lambda)$ is a function which will be chosen later. Assume that $\Phi$ attains its maximum value at point $u_0$. If $V$ is parallel to the normal direction $\nu$ at $u_0$, we have $\langle V,\nu\rangle=|V|$. Thus our result holds. So we assume $V$ is not parallel to the normal direction $\nu$ at $u_0$. We can choose the local orthonomal frame $\{E_1,\cdots,E_n\}$ on $\Sigma$ satisfying
\begin{equation*}
  \langle V, E_1\rangle\neq 0, \quad \mbox{and} \quad \langle V,
E_i\rangle=0, \quad  \forall ~ i\geq 2.
\end{equation*}
Obviously, $V=\langle V, E_1\rangle E_1+ \langle V, \nu\rangle \nu$.
Then, we arrive at $u_0$
\begin{eqnarray}\label{Par-1}
0=\Phi_i= - \frac{\nabla_{E_i}\tau}{\tau}+ \gamma^{\prime} \nabla_{E_i} \Lambda,
\end{eqnarray}
\begin{eqnarray}\label{Par-2}
0\geq \Phi_{ii}=- \frac{\nabla^2_{E_i, E_i}\tau}{\tau}+ \frac{|\nabla_{E_i}\tau|^2}{\tau^2}+ \gamma^{\prime} \nabla^2_{E_i,E_i} \Lambda+\gamma^{\prime\prime} |\nabla_{E_i} \Lambda|^2.
\end{eqnarray}
From  Lemma \ref{supp}, \eqref{Par-1} and \eqref{Par-2}, we have
\begin{equation*}
  \begin{split}
     0\geq  -\frac{1}{\tau}(-\tau h_{il}h_{li}+\lambda^{\prime} h_{ii}+(h_{iil}-\overline{R}_{0iil})\Lambda_l)
   +(\gamma^{\prime\prime}+(\gamma^{\prime})^2)\Lambda_i^2+\gamma^{\prime}(\lambda^{\prime} g_{ii}-\tau h_{ii}).
  \end{split}
\end{equation*}
By \eqref{1d-tau} and \eqref{Par-1}, we obtain
\begin{eqnarray}\label{Par-3}
h_{11}=\tau \gamma^{\prime}, \quad \quad h_{i1}=0, \quad \forall ~i\geq 2.
\end{eqnarray}
Therefore, it is possible to rotate the coordinate system such that $\{E_1, \cdots, E_n\}$ are the principal curvature directions of the second fundamental form $(h_{ij})$, i.e., $h_{ij}=h_{ii}\delta_{ij}.$ Thus, from   \eqref{1d-tau}-\eqref{2d-tau}, \eqref{Par-2} and \eqref{Par-3},  we get
\begin{eqnarray}\label{ineq-1}
    0 & \geq & F^{ii}h_{ii}^2-\frac{1}{\tau}\lambda^{\prime}F^{ii}h_{ii}-
    \frac{1}{\tau}F^{ii}(h_{iil}-\overline{R}_{0iil})\Lambda_l\\
    \nonumber&&+(\gamma^{\prime\prime}+(\gamma^{\prime})^2)
    F^{ii}\Lambda_i^2-\gamma^{\prime}\tau F^{ii}h_{ii}+\gamma^{\prime}\lambda^{\prime}F^{ii}g_{ii} \\
    \nonumber & = & F^{ii}h_{ii}^2-\frac{1}{\tau}F^{ii}h_{ii1}\Lambda_1+
    \frac{1}{\tau}F^{ii}\overline{R}_{0ii1}\Lambda_1+(\gamma^{\prime\prime}+
    (\gamma^{\prime})^2) F^{11}\Lambda_1^2\\
     \nonumber && +\gamma^{\prime}\lambda^{\prime}
    F^{ii}g_{ii}-\frac{1}{\tau}\lambda^{\prime}F^{ii}h_{ii}-\gamma^{\prime}\tau F^{ii}h_{ii}.
\end{eqnarray}

Since $\eta_{ii}= \sum_{j\neq i} h_{jj}$, then
%$\sum_i \eta_{ii} = (n-1) \sum_{i} h_{ii}, \quad h_{ii}= \frac{1}{n-1} \sum_k \eta_{kk}- \eta_{ii}.$
%It follows that
\begin{equation}\label{ht-c2-03}
\begin{aligned}
\sum_i F^{ii} h_{ii} =&\sum_i \left( \sum_k G^{kk} -G^{ii}\right) \left(\frac{1}{n-1} \sum_l \eta_{ll} -\eta_{ii}\right)\\
=& \sum_i G^{ii} \eta_{ii}\\
=&  \frac{1}{k-l} \left(\frac{\sigma_k(\eta)}{\sigma_l(\eta)}\right)^{\frac{1}{k-l}-1} \frac{\sum_i \eta_{ii}\sigma_{k-1}(\eta| i)\sigma_l(\eta)-\sigma_k(\eta) \sum_i \eta_{ii}\sigma_{l-1}(\eta| i)}{\sigma_l^2(\eta)}\\
=& \tilde{f},
\end{aligned}
\end{equation}
where $\tilde{f}=f^{\frac{1}{k-l}}$.
Note that the curvature equation \eqref{Eq}  can be written as
\begin{equation}\label{ht-eq-2}
G(\eta)=\tilde{f}.
\end{equation}
 Differentiating \eqref{ht-eq-2} with respect to $E_1$, we obtain
$$G^{ii} \eta_{ii1}= d_V\widetilde{f}(\nabla_{E_1}V)+h_{11}d_{\nu}\widetilde{f}(E_1).$$
In fact
\begin{eqnarray}\label{ht-c2-110}
\nonumber F^{ii}h_{ii1}&=&\sum_i\left (\sum_jG^{jj}-G^{ii}\right )h_{ii1}\\
%\nonumber&=&\left(\sum_jG^{jj} \right)\left(\sum_ih_{ii1}\right)-\sum_iG^{ii}h_{ii1}\\
\nonumber&=&\sum_iG^{ii}\eta_{ii1}\\
&=&d_V\widetilde{f}(\nabla_{E_1}V)+h_{11}d_{\nu}\widetilde{f}(E_1).
\end{eqnarray}
Putting \eqref{1d-lad}, \eqref{Par-3}, \eqref{ht-c2-03} and \eqref{ht-c2-110} into \eqref{ineq-1}, we derive
\begin{eqnarray*}
    \nonumber0 &\geq & F^{ii}h_{ii}^2-\frac{1}{\tau}F^{ii}(h_{ii1}-\overline{R}_{0ii1})\Lambda_1
    +(\gamma^{\prime\prime}+(\gamma^{\prime})^2)F^{11}\Lambda_1^2
    +\gamma^{\prime}\lambda^{\prime}F^{ii}g_{ii}-\frac{\lambda^{\prime}}{\tau}\widetilde{f}
    -\gamma^{\prime}\tau\widetilde{f}\\
    %\nonumber &= & F^{ii}h_{ii}^2 -\frac{1}{\tau}\left(d_V\widetilde{f}(\nabla_{E_1}V)+h_{11}d_{\nu}\widetilde{f}(E_1)\right)\Lambda_1+\frac{1}{\tau}
    %F^{ii}\overline{R}_{0ii1}\Lambda_1\\
    %\nonumber&&+(\gamma^{\prime\prime}+(\gamma^{\prime})^2)
     % F^{11}\Lambda_1^2+\gamma^{\prime}\lambda^{\prime}F^{ii}g_{ii}-\frac{\lambda^{\prime}}{\tau}\widetilde{f}
    %-\gamma^{\prime}\tau\widetilde{f}\\
    \nonumber&=&F^{ii}h_{ii}^2 -\frac{1}{\tau}d_V\widetilde{f}(\nabla_{E_1}V)\Lambda_1-\gamma^{\prime}d_{\nu}\widetilde{f}(E_1)\Lambda_1
    -\frac{1}{\tau} F^{ii}\overline{R}_{0i1i}\Lambda_1\\
   \nonumber &&+(\gamma^{\prime\prime}+(\gamma^{\prime})^2)
      F^{11}\Lambda_1^2+\gamma^{\prime}\lambda^{\prime}F^{ii}g_{ii}-\frac{\lambda^{\prime}}{\tau}\widetilde{f}
    -\gamma^{\prime}\tau\widetilde{f}\\
   \nonumber &=&F^{ii}h_{ii}^2 -\frac{1}{\tau}\left(\langle V,E_1\rangle d_V\widetilde{f}(\nabla_{E_1}V)+\lambda^{\prime}\widetilde{f}\right)-\gamma^{\prime}
    d_{\nu}\widetilde{f}(E_1)\langle V,E_1\rangle\\
    &&-\frac{1}{\tau} F^{ii}\overline{R}_{0i1i}\langle V,E_1\rangle
    +(\gamma^{\prime\prime}+(\gamma^{\prime})^2)
      F^{11}{\langle V,E_1\rangle}^2+\gamma^{\prime}\lambda^{\prime}F^{ii}g_{ii}
    -\gamma^{\prime}\tau\widetilde{f}.
\end{eqnarray*}
Since $V=\langle V,E_1\rangle E_1+\langle V,\nu\rangle \nu$, we have
\begin{eqnarray*}
  d_V\widetilde{f}(V,\nu)&=&\langle V,E_1\rangle d_V\widetilde{f}(\nabla_{E_1}V)+\langle V,\nu\rangle d_V\widetilde{f}(\nabla_{\nu}V).
\end{eqnarray*}
From \eqref{ASS3} and $V=\lambda\frac{\partial}{\partial r}$, we see that
\begin{eqnarray*}
  0&\geq & \frac{\partial}{\partial r}\left(\lambda^{k-l}f\right)=\frac{\partial}{\partial r}\left(\lambda^{k-l}\widetilde{f}^{k-l}\right)\\
  &=&(k-l)(\lambda\widetilde{f})^{k-l-1}\left(\lambda^{\prime}\widetilde{f}+d_V\widetilde{f}\right)\\
  &=&(k-l)(\lambda\widetilde{f})^{k-l-1}\left(\lambda^{\prime}\widetilde{f}+\langle V,E_1\rangle d_V\widetilde{f}(\nabla_{E_1}V)+\langle V,\nu\rangle d_V\widetilde{f}(\nabla_{\nu}V)\right).
\end{eqnarray*}
It follows that
$$-\left(\lambda^{\prime}\widetilde{f}+\langle V,E_1\rangle d_V\widetilde{f}(\nabla_{E_1}V)\right)\geq
\langle V,\nu\rangle d_V\widetilde{f}(\nabla_{\nu}V),$$
which implies
\begin{eqnarray}\label{eqam}
  \nonumber 0 &\geq& F^{ii}h_{ii}^2 +d_V\widetilde{f}(\nabla_{\nu}V)-\gamma^{\prime}
    d_{\nu}\widetilde{f}(E_1)\langle V,E_1\rangle
    -\frac{1}{\tau} F^{ii}\overline{R}_{0i1i}\langle V,E_1\rangle\\
    &&+(\gamma^{\prime\prime}+(\gamma^{\prime})^2)
      F^{11}{\langle V,E_1\rangle}^2+\gamma^{\prime}\lambda^{\prime}F^{ii}g_{ii}
    -\gamma^{\prime}\tau\widetilde{f}.
\end{eqnarray}
Choosing the function $\gamma(r)=\frac{\alpha}{r}$ for a positive constant $\alpha$, we get
\begin{equation}\label{eqna-2}
  \gamma^{\prime}(r)=-\frac{\alpha}{r^2},\quad\quad  \gamma^{\prime\prime}(r)=\frac{2\alpha}{r^3}.
\end{equation}
By \eqref{Par-3} and the choice of function $\gamma(r)$, we have $h_{11}<0$ at $u_0$. From $H>0$ we know that
\begin{equation}\label{F11}
   F^{11}=\sum_{j \neq 1} G^{jj} \geq \frac{1}{2} \sum_i G^{ii} =\frac{1}{2(n-1)}\sum_iF^{ii}\geq \frac{1}{2}\left(\frac{C_n^k}{C_n^l}\right)^{\frac{1}{k-l}}.
\end{equation}
 Putting \eqref{eqna-2} into \eqref{eqam}, we have
 \begin{eqnarray}\label{eqmm}
  \nonumber 0&\geq& F^{ii}h_{ii}^2 +d_V\widetilde{f}(\nabla_{\nu}V)+\frac{\alpha}{r^2}
    d_{\nu}\widetilde{f}(E_1)\langle V,E_1\rangle
    -\frac{1}{\tau} F^{ii}\overline{R}_{0i1i}\langle V,E_1\rangle\\
   \nonumber &&+(\frac{2\alpha}{r^3}+\frac{\alpha^2}{r^4})
      F^{11}{\langle V,E_1\rangle}^2-\frac{\alpha}{r^2}\lambda^{\prime}F^{ii}g_{ii}
    +\frac{\alpha}{r^2}\tau\widetilde{f}\\
   \nonumber &\geq &\tau^2\frac{\alpha^2}{r^4}F^{11}+(\frac{2\alpha}{r^3}+\frac{\alpha^2}{r^4})
      F^{11}{\langle V,E_1\rangle}^2 +\frac{\alpha}{r^2}
    d_{\nu}\widetilde{f}(E_1)\langle V,E_1\rangle\\
    \nonumber&&-\frac{1}{\tau} F^{ii}\overline{R}_{0i1i}\langle V,E_1\rangle
    +d_V\widetilde{f}(\nabla_{\nu}V)
    -\frac{\alpha}{r^2}\lambda^{\prime}F^{ii}g_{ii}
    +\frac{\alpha}{r^2}\tau\widetilde{f}\\
    \nonumber&=&\frac{\alpha^2}{r^4}F^{11}|V|^2+\frac{2\alpha}{r^3}F^{11}\langle V,E_1\rangle^2+\frac{\alpha}{r^2}
    d_{\nu}\widetilde{f}(E_1)\langle V,E_1\rangle\\
    &&-\frac{1}{\tau} F^{ii}\overline{R}_{0i1i}\langle V,E_1\rangle
    +d_V\widetilde{f}(\nabla_{\nu}V)
    -\frac{\alpha}{r^2}\lambda^{\prime}\sum_iF^{ii}
    +\frac{\alpha}{r^2}\tau\widetilde{f},
 \end{eqnarray}
 the third equality comes from $|V|^2=\langle V,E_1\rangle^2+\langle V,\nu\rangle^2$.

 Since $V=\langle V,E_1\rangle E_1+\langle V,\nu\rangle\nu$, we can find that $V \bot Span\{E_2,\cdots,E_n\}$. On the other hand, $E_1,\nu \bot Span\{E_2,\cdots,E_n\}$. It is possible to choose coordinate systems such that $\overline{e}_1\bot Span\{E_2,\cdots,E_n\}$, which
implies that the pair $\{V,\overline{e}_1\}$ and $\{\nu,E_1\}$ lie in the same plane and
$$Span\{E_2,\cdots,E_n\}=Span\{\overline{e}_2,\cdots,\overline{e}_n\}.$$
Therefore, we can choose $E_2 = \overline{e}_2 ,\cdots,E_n = \overline{e}_n$. The vector $\nu$ and $E_1$ can be decomposed into
\begin{equation}\label{R-1}
  \nu=\langle\nu,\overline{e}_0\rangle \overline{e}_0+\langle\nu,\overline{e}_1\rangle \overline{e}_1
=\frac{\tau}{\lambda}\overline{e}_0+\langle\nu,\overline{e}_1\rangle \overline{e}_1,
\end{equation}
\begin{equation}\label{R-2}
 E_1=\langle E_1,\overline{e}_0\rangle \overline{e}_0+\langle E_1,\overline{e}_1\rangle \overline{e}_1.
\end{equation}
By \eqref{R-1} and \eqref{R-2}, we obtain
\begin{eqnarray}\label{eqnt}
  \overline{R}_{0i1i} &=& \overline{R}(\nu,E_i,E_1,E_i)\\
  \nonumber&=& \frac{\tau}{\lambda}\langle E_1,\overline{e}_0\rangle\overline{R}(\overline{e}_0,\overline{e}_i,
  \overline{e}_0,\overline{e}_i)+\langle\nu,\overline{e}_1\rangle\langle E_1,\overline{e}_1\rangle\overline{R}(\overline{e}_1,
  \overline{e}_i,\overline{e}_1,\overline{e}_i)\\
  \nonumber&=&\frac{\tau}{\lambda}\langle E_1,\overline{e}_0\rangle\overline{R}(\overline{e}_0,\overline{e}_i,
  \overline{e}_0,\overline{e}_i)-\tau\frac{\langle\nu,\overline{e}_1\rangle^2}{\langle E_1,V\rangle}
  \overline{R}(\overline{e}_1,\overline{e}_i,\overline{e}_1,\overline{e}_i)\\
  \nonumber&=&\tau\left(\frac{1}{\lambda}\langle E_1,\overline{e}_0\rangle\overline{R}(\overline{e}_0,\overline{e}_i,
  \overline{e}_0,\overline{e}_i)-\frac{\langle\nu,\overline{e}_1\rangle^2}{\langle E_1,V\rangle}
  \overline{R}(\overline{e}_1,\overline{e}_i,\overline{e}_1,\overline{e}_i)\right),
\end{eqnarray}
the second equality comes from $0=\overline{R}_{ijk0}$ (see \cite[Lemma 2.1]{CLW}), the third equality comes from $0=\langle V,\overline{e}_1\rangle$.

From \eqref{F11} and \eqref{eqnt}, \eqref{eqmm} becomes
\begin{eqnarray*}
 \nonumber 0 &\geq&
   C_1F^{11}\frac{\alpha^2}{r_2^4}-C_2F^{11}\frac{\alpha}{r_2^3}-C_3\frac{\alpha}{r_1^2}|d_{\nu}\widetilde{f}
  (E_1)|-C_4F^{11}-|d_V\widetilde{f}(\nabla_{\nu}V)|-C_5\\
   &\geq& C\alpha^2F^{11}-C_2\alpha F^{11}-C\alpha|d_{\nu}\widetilde{f}
  (E_1)|-CF^{11}-|d_V\widetilde{f}(\nabla_{\nu}V)|-C,
\end{eqnarray*}
where $r_1=\inf_{\Sigma}r$, $r_2=\sup_{\Sigma}r$, $C_1$, $C_2$, $C_3$, $C_4$, $C_5$, $C$ depend on $n$, $r_1$, $r_2$, $\inf_{\Sigma}f$, the $C^1$ bounds of $\lambda$ and curvature $\overline{R}$. Thus, we have a contradiction when $\alpha$ is large enough. Hence, $V$ is parallel to the
normal $\nu$ which implies the lower bound of $\tau$.
\end{proof}

%%%%%%%%%%%%%%%%%%%%%%%%%%%%%%%%%%%%%%%%%%%%%%%%%%%%%%%%%%%%%%%%%%%%%%%%%%%%%%%
\subsection{$C^2$ Estimates}
%%%%%%%%%%%%%%%%%%%%%%%%%%%%%%%%%%%%%%%%%%%%%%%%%%%%%%%%%%%%%%%%%%%%%%%%%%%%%%

  Under the  the assumption \eqref{ASS1}-\eqref{ASS3},  from Theorem   \ref{n-2-C^0} and  Theorem  \ref{n-2-C1e} we know that
there exists a positive constant $C$ depending on $\inf_{\Sigma} r$ and $\|r\|_{C^1}$ such that
$$\frac{1}{C} \leq \inf_{\Sigma} \tau \leq
\tau \leq \sup_{\Sigma} \tau \leq C.$$

\begin{theorem}\label{n-2-C2e}
Let $\Sigma$ be a closed star-shaped  $(\eta,k)$-convex hypersurface satisfying the curvature equation \eqref{Eq2}  and the assumption of Theorem \ref{Main}. Then, there exists a constant C depending only on $n,k,l,\inf_{\Sigma}\lambda',\inf_{\Sigma}r,\inf_{\Sigma}f,\|r\|_{C^2}$,$\|f\|_{C^2}$ and the curvature $\overline{R}$ such that for $1\leq i\leq n$
\begin{equation*}
|\kappa_{i}(u)|\le C, \quad \forall ~ u \in M.
\end{equation*}
\end{theorem}

\begin{proof}
Since $\eta\in\Gamma_{k}\subset\Gamma_{1}$, we see that the mean curvature is positive. It suffices to prove that the largest curvature $\kappa_{\mbox{max}}$ is uniformly bounded from above. Take the auxiliary function
\begin{equation*}
  P=\ln\kappa_{max}-\ln(\tau-a)+A\Lambda,
\end{equation*}
where $a=\frac{1}{2}\inf_{\Sigma}(\tau)$ and $A>1$ is a constant to be determined later. Assume that $P$ attains its maximum value at point $u_0$. We can choose a local orthonormal frame  $\{E_{1}, E_{2}, \cdots, E_{n}\}$ near $u_0$ such that
$$h_{ii}=\delta_{ij}h_{ij}, \quad  h_{11}\geq h_{22}\geq \cdots \geq h_{nn}$$
at $u_0$. Recalling that $\eta_{ii}=\sum_{k\neq i}h_{kk}$, we have
$$\eta_{11}\leq \eta_{22}\leq\cdots\leq\eta_{nn}.$$
It follows that
$$G^{11}\geq G^{22}\geq\cdots\geq G^{nn},\quad F^{11}\leq F^{22}\leq\dots\leq F^{nn}.$$
We define a new function $Q$ by
$$ Q=\ln h_{11}-\ln(\tau-a)+A\Lambda.$$
Since $h_{11}(u_0)=\kappa_{\mbox{max}}(u_0)$ and $h_{11}\leq\kappa_{\mbox{max}}$ near $u_0$, $Q$ achieves a maximum at $u_0$.

Hence
\begin{equation}\label{ht-c2-01}
0=Q_i=\frac{h_{11i}}{h_{11}}-\frac{\tau_i}{\tau-a}+A\Lambda_i,
\end{equation}
\begin{equation}\label{ht-c2-02}
0\geq F^{ii}Q_{ii}=F^{ii}(\ln h_{11})_{ii}-F^{ii}(\ln(\tau-a))_{ii}+AF^{ii}\Lambda_{ii}.
\end{equation}

We divide our proof into four steps.

\textbf{Step 1}:  We claim that
\begin{eqnarray}\label{ht-c2-1}
  0 &\geq& -\frac{2}{h_{11}}\sum_{i\geq2}G^{1i,i1}h_{1i1}^2 -\frac{F^{ii}h_{11i}^2}{h_{11}^2}
  +\frac{aF^{ii}h_{ii}^2}{\tau-a}
  +F^{ii}\frac{\tau_i^2}{(\tau-a)^2}\\
  \nonumber&&+(A\lambda'-C_0)\sum_iF^{ii}-C_0h_{11}-\frac{C_0(1+\sum_iF^{ii})}{h_{11}} -AC_0,
\end{eqnarray}
where $C_0$ depends on $\inf_{\Sigma}r,\inf_{\Sigma}f,\|r\|_{C^2},\|f\|_{C^2}$ and the curvature $\overline{R}$.

Using the similar argument in \eqref{ht-c2-110}, we obtain
\begin{equation}\label{eqe}
  F^{ii}h_{iij}=d_V\widetilde{f}(\nabla_{E_j}V)+h_{jj}d_{\nu}\widetilde{f}(E_j).
\end{equation}
By Gauss formula and Weingarten formula,
\begin{equation}\label{tau}
  \tau_i=h_{ii}\langle V,E_i\rangle, \quad \tau_{ii}=\sum_j h_{iji}\langle V, E_j\rangle-\tau h_{ii}^2+h_{ii}.
\end{equation}
Combined with \eqref{eqe}, \eqref{tau} and Codazzi formula, we have
\begin{eqnarray}\label{ht-c2-05}
\nonumber-F^{ii}(\ln(\tau-a))_{ii}&=&-F^{ii}(\frac{\tau_{ii}}{\tau-a}-\frac{\tau_i^2}{(\tau-a)^2})\\
%\nonumber&=&-\frac{1}{\tau-a}\sum_{i,j} F^{ii}h_{iji}\langle V,E_j\rangle+\frac{\tau F^{ii}h_{ii}^2}{\tau-a}-\frac{1}{\tau-a}
%F^{ii}h_{ii}+F^{ii}\frac{\tau_i^2}{(\tau-a)^2}\\
%\nonumber&=&-\frac{1}{\tau-a}\sum_{i,j} F^{ii}h_{iij}\langle V,E_j\rangle-\frac{1}{\tau-a}\sum_{i,j}\overline{R}_{0iji}F^{ii}\langle V,E_j\rangle+\frac{\tau F^{ii}h_{ii}^2}{\tau-a}\\
%\nonumber&&-\frac{1}{\tau-a}F^{ii}h_{ii}+F^{ii}\frac{\tau_i^2}{(\tau-a)^2}\\
\nonumber&=&-\frac{1}{\tau-a}\sum_jh_{jj}(d_{\nu}\widetilde{f})(E_j)\langle V,E_j\rangle
-\frac{1}{\tau-a}\sum_jd_V\widetilde{f}(\nabla_{E_j}V)\langle V,E_j\rangle\\
\nonumber&&-\frac{1}{\tau-a}\sum_{i,j}\overline{R}_{0iji}F^{ii}\langle V,E_j\rangle+\frac{\tau F^{ii}h_{ii}^2}{\tau-a}-\frac{1}{\tau-a}
\widetilde{f}+F^{ii}\frac{\tau_i^2}{(\tau-a)^2}\\
&\geq&-\frac{1}{\tau-a}\sum_jh_{jj}(d_{\nu}\widetilde{f})(E_j)\langle V,E_j\rangle+\frac{\tau F^{ii}h_{ii}^2}{\tau-a}
+F^{ii}\frac{\tau_i^2}{(\tau-a)^2}-C_1\sum_iF^{ii}-C_1,
\end{eqnarray}
where $C_1$ depends on $\inf_{\Sigma}r,\|r\|_{C^1}$, $\|f\|_{C^1}$ and the curvature $\overline{R}$.

Differentiating \eqref{ht-eq-2} with respect to $E_1$ twice, we obtain
$$G^{ij}\eta_{ij1}=d_V\widetilde{f}(\nabla_{E_1}V)+h_{1k}d_{\nu}\widetilde{f}(E_k)$$
and
\begin{eqnarray*}
   \nonumber G^{ij,rs}\eta_{ij1}\eta_{rs1}+G^{ij}\eta_{ij11}  &=&d_V^2\widetilde{f}(\nabla_{E_1}V,\nabla_{E_1}V)+d_V\widetilde{f}(\nabla_{E_1,E_1}^2V) \\
 \nonumber &&+2d_Vd_{\nu}\widetilde{f}
  (\nabla_{E_1}V,\nabla_{E_1}\nu)+d_{\nu}^2\widetilde{f}(\nabla_{E_1}\nu,\nabla_{E_1}\nu)+d_{\nu}\widetilde{f}(\nabla_{E_1,E_1}^2\nu)\\
  &\geq& \sum_ih_{1i1}(d_{\nu}\widetilde{f})(E_i)-C_2h_{11}^2-C_2h_{11}-C_2.
\end{eqnarray*}
Applying the concavity of $G$, we derive
$$-G^{ij,rs} \eta_{ij1} \eta_{rs1} \geq -2 \sum_{i\geq2} G^{1i, i1} \eta_{1i1}^2=-2 \sum_{i\geq 2} G^{1i, i1} h_{1i1}^2.$$
%From Codazzi equation and Cauchy-Schwarz inequality,
%$$-2\sum_{i\geq2}G^{1i,i1}h_{1i1}^2=-2\sum_{i\geq2}G^{1i,i1}(h_{11i}+\overline{R}_{01i1})^2
%\geq-2\sum_{i\geq2}(1-\epsilon)G^{1i,i1}h_{11i}^2-2C_{\varepsilon}\sum_{i\geq2}G^{1i,i1}.$$
%then we obtain
%$$-G^{ij,rs} \eta_{ij1} \eta_{rs1}\geq-2\sum_{i\geq2}G^{1i, i1}\left((1-\varepsilon)h_{11i}^2+C_{\varepsilon}\right)$$
%where $\varepsilon<1$ is a small constant to be determined later.
It follows that
   \begin{equation}\label{Fii}
     F^{ii}h_{ii11}=G^{ii}\eta_{ii11}\geq  -2\sum_{i\geq2}G^{1i,i1}h_{1i1}^2+ \sum_i h_{1i1} (d_{\nu} \widetilde{f})(E_i) -C_2h_{11}^2-C_2h_{11}-C_2,
   \end{equation}
where $C_2$ depends on $\|f\|_{C^2}$, $\|r\|_{C^2}$.
%Recalling that
%\begin{equation*}
%     h_{11ii}=h_{ii11}-h_{11}h_{ii}^2+h_{11}^2h_{ii}-2(h_{ii}-h_{11})\overline{R}_{i1i1}-h_{11}\overline{R}
 %    _{i0i0}+h_{ii}\overline{R}_{1010}-\overline{R}_{i1i0;1}+\overline{R}_{1i10;i}.
  % \end{equation*}

Combined with \eqref{hii}, \eqref{Fii} and Codazzi formula, we have
\begin{eqnarray}\label{ht-c2-06}
\nonumber F^{ii}(\ln h_{11})_{ii}&=&\frac{F^{ii}h_{11ii}}{h_{11}}-\frac{F^{ii}h_{11i}^2}{h_{11}^2}\\
%\nonumber&=&\frac{F^{ii}}{h_{11}} \left(h_{ii11} -h_{11}h_{ii}^2+h_{11}^2h_{ii}\right)+\frac{F^{ii}}{h_{11}}\left(\overline{R}_{1i10;i}-
%\overline{R}_{i1i0;1}\right)-\frac{2F^{ii}h_{ii}}{h_{11}}\overline{R}_{i1i1}\\
%\nonumber &&+2F^{ii}\overline{R}_{i1i1}-F^{ii}\overline{R}_{i0i0}+\frac{F^{ii}h_{ii}}{h_{11}}\overline{R}_{1010}
%-\frac{F^{ii}h_{11i}^2}{h_{11}^2}\\
\nonumber&\geq& -\frac{2}{h_{11}}\sum_{i\geq2}G^{1i,i1}h_{1i1}^2 +\frac{1}{h_{11}}
\sum_i h_{11i} (d_{\nu}\widetilde{f})(E_i)+\frac{1}{h_{11}}\sum_i\overline{R}_{01i1}(d_{\nu}\widetilde{f})(E_i)\\
\nonumber&&+h_{11}\widetilde{f}-F_{ii}h_{ii}^2
+\frac{F^{ii}}{h_{11}}(\overline{R}_{1i10;i}-\overline{R}_{i1i0;1})+2F^{ii}\overline{R}
_{i1i1}-F^{ii}\overline{R}_{i0i0}\\
\nonumber&&
-C_2\sum_iF^{ii}+\frac{\widetilde{f}}{h_{11}}\overline{R}_{1010}
-\frac{F^{ii}h_{11i}^2}{h_{11}^2}-C_2h_{11}-\frac{C_2}{h_{11}}-C_2\\
&\geq&- \frac{2}{h_{11}}\sum_{i\geq 2}G^{1i, i1}h_{1i1}^2  +\frac{1}{h_{11}}
\sum_i h_{11i} (d_{\nu}\widetilde{f})(E_i)-F^{ii}h_{ii}^2\\
\nonumber&&-\frac{F^{ii}h_{11i}^2}{h_{11}^2}-C_3h_{11}-\frac{C_3}{h_{11}}-\frac{C_3\sum_iF^{ii}}{h_{11}}-C_3\sum_iF^{ii}-C_3,
\end{eqnarray}
where $C_3$ depends on $\inf_{\Sigma}f,\|f\|_{C^2}$, $\|r\|_{C^2}$ and the curvature $\overline{R}$.

By \eqref{2d-lad}, we derive
\begin{equation}\label{AFii}
  AF^{ii}\Lambda_{ii}=A\lambda'F^{ii}g_{ii}-A\tau F^{ii}h_{ii}=A\lambda'\sum_iF^{ii}-A\tau\widetilde{f}.
\end{equation}
Taking \eqref{ht-c2-05}, \eqref{ht-c2-06} and \eqref{AFii} into \eqref{ht-c2-02}, we get
\begin{eqnarray*}
  0 &\geq& -\frac{2}{h_{11}}\sum_{i\geq2}G^{1i,i1}h_{1i1}^2 -\frac{F^{ii}h_{11i}^2}{h_{11}^2}
  +\frac{1}{h_{11}}\sum_ih_{11i}(d_{\nu}\widetilde{f})(E_i)\\
 \nonumber&& -\frac{1}{\tau-a} \sum_jh_{jj}(d_{\nu}\widetilde{f})(E_j)\langle V,E_j\rangle+\frac{aF^{ii}h_{ii}^2}{\tau-a}
  +F^{ii}\frac{\tau_i^2}{(\tau-a)^2}\\
  \nonumber&&+(A\lambda'-C_4)\sum_iF^{ii}-C_4h_{11}-\frac{C_4(1+\sum_iF^{ii})}{h_{11}}
  -AC_4.
\end{eqnarray*}
By \eqref{ht-c2-01}, \eqref{tau},
\begin{eqnarray*}
   &&\frac{1}{h_{11}}\sum_ih_{11i}(d_{\nu}\widetilde{f})(E_i)-\frac{1}{\tau-a} \sum_jh_{jj}(d_{\nu}\widetilde{f})(E_j)\langle V,E_j\rangle \\
  \nonumber&=&\sum_i\left(\frac{h_{11i}}{h_{11}}-\frac{\tau_i}{\tau-a}\right)(d_{\nu}\widetilde{f})(E_i)\\
  \nonumber&=&-A\sum_i(d_{\nu}\widetilde{f})(E_i)\langle V,E_i\rangle \\
  \nonumber&\geq&-AC_4,
\end{eqnarray*}
which implies the inequality \eqref{ht-c2-1}.

\textbf{Step 2}:
There exists a positive constant $\delta<\frac{1}{n-2}$ such that
$$\frac{C_{n-1}^{k-1} [1-(n-2)\delta]^{k-1} -(n-1) \delta C_{n-1}^{k-2} [1+(n-2)\delta]^{k-2} }{C_n^l [1+(n-2)\delta]^l } >\frac{C_{n-1}^{k-1}}{2C_n^l}.$$

We claim that there exists a constant $B>1$ depending on $n ,k, l, \delta$, $\inf_{\Sigma} r$, $\inf_{\Sigma} f$, $\|r\|_{C^2}$, $\|f\|_{C^2}$ and the curvature $\overline{R}$, such that
\begin{equation*}
\frac{aF^{ii} h_{ii}^2}{2(\tau-a)}+\frac{A\lambda'-C_0}{2}\sum_i F^{ii}\geq C_0h_{11},
\end{equation*}
if  $h_{11} \geq B$, $A= \left( 4\|f\|_{C^0}^{1-\frac{1}{k-l}}  \frac{k C_n^l}{(n-k+1) C_{n-1}^{k-1}\inf\lambda'} +\frac{27}{\inf\lambda'}  \right) C_0$.

Case 1: $|h_{ii}|\leq \delta h_{11}$ for all $i\geq 2$.\\
In this case, we have
\begin{equation*}
|\eta_{11}| \leq (n-1) \delta h_{11}, \quad [1-(n-2)\delta]h_{11}\leq \eta_{22}\leq \cdots \leq \eta_{nn}\leq [1+(n-2)\delta]h_{11}.
\end{equation*}
By the definitions of $G^{ii}$ and $F^{ii}$, we obtain
\begin{equation*}
\begin{aligned}
\sum_i F^{ii}&=(n-1)\sum_i G^{ii}\\
&= \frac{n-1}{k-l} \left(\frac{\sigma_k(\eta)}{\sigma_l(\eta)}\right)^{\frac{1}{k-l}-1} \frac{(n-k+1)\sigma_{k-1}(\eta)\sigma_l(\eta)-(n-l+1)\sigma_k(\eta)\sigma_{l-1}(\eta)}{\sigma_l^2(\eta)}\\
&\geq \frac{C_n^k}{C_n^{k-1}} \left(\frac{\sigma_k(\eta)}{\sigma_l(\eta)}\right)^{\frac{1}{k-l}-1} \left(\frac{\sigma_{k-1}(\eta)}{\sigma_l(\eta)}\right)  \\
&= \frac{C_n^k}{C_n^{k-1}} \left(\frac{\sigma_k(\eta)}{\sigma_l(\eta)}\right)^{\frac{1}{k-l}-1} \left(\frac{
\sigma_{k-1}(\eta|1)+\eta_{11}\sigma_{k-2}(\eta|1)}{\sigma_l(\eta)}\right)  \\
&\geq \frac{n-k+1}{k} f^{\frac{1}{k-l}-1} \frac{C_{n-1}^{k-1} [1-(n-2)\delta]^{k-1} -(n-1) \delta C_{n-1}^{k-2} [1+(n-2)\delta]^{k-2} }{C_n^l [1+(n-2)\delta]^l } h_{11}^{k-1-l}\\
&\geq f^{\frac{1}{k-l}-1} \frac{ (n-k+1)C_{n-1}^{k-1}}{2kC_n^l  } h_{11},
\end{aligned}
\end{equation*}
 which implies that
$$C_0h_{11} \leq \frac{A\lambda'-27C_0}{2} \sum_i F^{ii}.$$

Case 2: $h_{22} > \delta h_{11}$ or $h_{nn} <- \delta h_{11}$.\\
In this case, we have
 \begin{equation*}
\begin{aligned}
\frac{a F^{ii} h_{ii}^2}{2(\tau-a)}&\geq \frac{a}{2(\sup \tau-a)} \left(F^{22} h_{22}^2+F^{nn} h_{nn}^2\right)\\
&\geq  \frac{a\delta^2}{2(\sup \tau-a)} F^{22} h_{11}^2\\
&\geq  \frac{a\delta^2}{2n(\sup \tau-a)} \sum_i G^{ii}   h^2_{11} \\
&\geq  \left(\frac{C_n^k}{C_n^l}\right)^{\frac{1}{k-l}} \frac{a\delta^2 h_{11}}{2n(\sup \tau-a)}   h_{11}.
\end{aligned}
\end{equation*}
Then, we conclude that
$$\frac{a F^{ii} h_{ii}^2}{2(\tau-a)}\geq C_0 h_{11},$$
if
 $$h_{11} \geq \left(\left(\frac{C_n^k}{C_n^l}\right)^{\frac{1}{k-l}} \frac{a\delta^2}{2n(\sup \tau-a)} \right)^{-1}C_0.$$

\textbf{Step 3}:   We claim that
$$|h_{ii}|\leq C_5A, \quad \forall~i\geq2,$$
 if  $h_{11} \geq B>1$, where $C_5$ is a constant depending on $n, k, l$, $\inf_{\Sigma} r$, $\inf_
 {\Sigma}f$, $\|r\|_{C^2}$, $\|f\|_{C^2}$ and the curvature $\overline{R}$.

Combined with Step 1 and Step 2, we obtain
\begin{eqnarray}\label{ht-c2-32}
\nonumber0&\geq& - \frac{2}{h_{11}} \sum_{i\geq 2}G^{1i, i1}h_{1i1}^2 -\frac{F^{ii}h_{11i}^2}{h_{11}^2}
+\frac{aF^{ii} h_{ii}^2}{2(\tau-a)}+\frac{F^{ii}{\tau}_i^2}{(\tau-a)^2}\\
&& + \frac{A\lambda'-C_0}{2}\sum_i F^{ii}-\frac{C_0(1+\sum_iF^{ii})}{h_{11}}-AC_0.
\end{eqnarray}
Use \eqref{ht-c2-01}, the concavity of $G$ and Cauchy-Schwarz inequality,
\begin{eqnarray*}
0&\geq&  -\frac{1+\epsilon}{(\tau-a)^2} F^{ii} \tau_i^2 - (1+ \frac{1}{\epsilon})A^2 F^{ii} \Lambda_i^2\\
&&+\frac{aF^{ii} h_{ii}^2}{2(\tau-a)}+\frac{F^{ii}\tau_i^2}{(\tau-a)^2} + \frac{A\lambda'-C_0}{2}\sum_i F^{ii}-C_0(1+\sum_iF^{ii})-AC_0\\
&\geq& \left(\frac{a}{2(\tau-a)}-\frac{C_0\epsilon}{(\tau-a)^2}\right) F^{ii}h^2_{ii}- \left((1+ \frac{1}{\epsilon})A^2C_0- \frac{A\lambda'-C_0}{2}\right) \sum_i F^{ii}- \left(\sum_iF^{ii}+A+1\right)C_0,
\end{eqnarray*}
where $\tau_i=h_{ii} \langle V, E_i\rangle$ in the second inequality.
Choose $\epsilon=\frac{a(\tau-a)}{4C_0}$,
\begin{equation}\label{ht-c2-31}
\begin{aligned}
0&\geq \frac{a}{4(\tau-a)} F^{ii}h^2_{ii}- \left((1+ \frac{4C_0}{a(\tau-a)})A^2C_0- \frac{A\lambda'-C_0}{2}\right) \sum_i F^{ii}-\left(\sum_iF^{ii}+A+1\right) C_0\\
&\geq \frac{a}{4(\sup \tau-a)} F^{ii}h^2_{ii}-\left((1+ \frac{4C_0}{a^2})A^2C_0- \frac{A\lambda'-C_0}{2}\right) \sum_i F^{ii}-\left(\sum_iF^{ii}+A+1\right) C_0.
\end{aligned}
\end{equation}
By \eqref{Fi} and \eqref{ht-c2-31}, we have
\begin{eqnarray*}
 0&\geq& \frac{a}{4(\sup \tau-a)n(n-1)}  \left(\sum_{k\geq 2} h_{kk}^2\right) \sum_iF^{ii}\\
&&-\left((1+ \frac{4C_0}{a^2})A^2C_0- \frac{A\lambda'-C_0}{2}+C_0+ \frac{(A+1)C_0}{(n-1)}\left(\frac{C_n^k}{C_n^l}\right)^{-\frac{1}{k-l}}\right) \sum_i F^{ii},
\end{eqnarray*}
which implies that
$$\sum_{k\geq 2} h_{kk}^2 \leq C_5^2A^2,$$
where $C_5$ is a constant depending on $n, k, l$, $\inf_{\Sigma} r$, $\inf _{\Sigma}f$, $\|r\|_{C^2}$, $\|f\|_{C^2}$ and the curvature $\overline{R}$.

\textbf{Step 4}:
We claim that there exists a constant $C$ depending on $n, k, l$, $\inf_{\Sigma}\lambda'$, $\inf_{\Sigma} r$, $\inf _{\Sigma}f$, $\|r\|_{C^2}$, $\|f\|_{C^2}$ and the curvature $\overline{R}$ such that $$h_{11}\leq C.$$
Without loss of generality, we assume that
\begin{equation}\label{ab}
  h_{11} \geq \max \left\{B, \left(\frac{32n C_0 A^2 (\sup \tau -a)}{\varepsilon a}\right)^{\frac{1}{2}}, \frac{C_5 A}{\beta}\right\},
\end{equation}
where $\beta<\frac{1}{2}$ will be determined later.
Recalling $\tau_1=h_{11} \langle V, E_1\rangle$, by  \eqref{ht-c2-01} and Cauchy-Schwarz inequality, we have
\begin{equation*}
\begin{aligned}
\frac{F^{11} h_{111}^2}{h_{11}^2}&\leq \frac{1+\varepsilon}{(\tau-a)^2} F^{11} \tau_1^2+(1+\frac{1}{\varepsilon}) A^2 F^{11}\langle V, E_1\rangle^2\\
&\leq \frac{F^{11} \tau_1^2}{(\tau-a)^2} +\frac{C_0\varepsilon F^{11} h_{11}^2}{(\tau-a)^2}  +(\frac{1+\varepsilon}{\varepsilon}) C_0A^2 F^{11}.
\end{aligned}
\end{equation*}
Choose $\varepsilon\leq \min\{\frac{a(\tau-a)}{16C_0},1\}$, such that
$$\frac{F^{11} h_{111}^2}{h_{11}^2}\leq \frac{F^{11} \tau_1^2}{(\tau-a)^2} +\frac{a F^{ii} h_{ii}^2}{16(\tau-a)}  +\frac{2C_0A^2 F^{11}}{\varepsilon}.$$
 Hence according to Step 3 and \eqref{ab}, we know that
\begin{equation}\label{ht-c2-41}
\begin{aligned}
\frac{F^{11} h_{111}^2}{h_{11}^2}\leq \frac{F^{11} \tau_1^2}{(\tau-a)^2} +\frac{a F^{ii} h_{ii}^2}{8(\tau-a)}
\end{aligned}
\end{equation}
and
$$|h_{ii}|\leq \beta h_{11}, \quad \forall i\geq 2.$$
Thus
$$\frac{1-\beta}{h_{11}-h_{ii}}\leq\frac{1}{h_{11}}\leq \frac{1+\beta}{h_{11}-h_{ii}}.$$
Combined with Proposition \ref{th-lem-07}, we obtain
\begin{eqnarray}\label{ht-c2-42}
\nonumber\sum_{i\geq 2} \frac{F^{ii}h_{11i}^2}{h_{11}^2}
&=&\sum_{i\geq 2} \frac{F^{ii}-F^{11}}{h_{11}^2} h_{11i}^2 +\sum_{i\geq 2} \frac{F^{11}h_{11i}^2}{h_{11}^2}\\
\nonumber&\leq& \frac{1+\beta}{h_{11}} \sum_{i\geq 2} \frac{F^{ii}-F^{11}}{h_{11}-h_{ii}} h_{11i}^2+\sum_{i\geq 2} \frac{F^{11}h_{11i}^2}{h_{11}^2}\\
\nonumber&=&\frac{1+\beta}{h_{11}} \sum_{i\geq 2} \frac{G^{11}-G^{ii}}{\eta_{ii}-\eta_{11}} h_{11i}^2+\sum_{i\geq 2} \frac{F^{11}h_{11i}^2}{h_{11}^2}\\
&=&-\frac{1+\beta}{h_{11}} \sum_{i\geq 2}G^{1i, i1} h_{11i}^2 +\sum_{i\geq 2} \frac{F^{11}h_{11i}^2}{h_{11}^2}.
\end{eqnarray}
%From Codazzi equation and Cauchy-Schwarz inequality,
%$$-\sum_{i\geq2}G^{1i,i1}h_{11i}^2=-\sum_{i\geq2}G^{1i,i1}(h_{1i1}+\overline{R}_{011i})^2
%\leq-\sum_{i\geq2}(1+\epsilon)G^{1i,i1}h_{1i1}^2-C_{\varepsilon}\sum_{i\geq2}G^{1i,i1}.$$
%then we obtain
%$$-G^{ij,rs} \eta_{ij1} \eta_{rs1}\geq-2\sum_{i\geq2}G^{1i, i1}\left((1-\varepsilon)h_{11i}^2+C_{\varepsilon}\right)$$
%where $\varepsilon<1$ is a small constant to be determined later.
Use  \eqref{ht-c2-01}, \eqref{ab}, Cauchy-Schwarz inequality and the fact $\tau_i=h_{ii}\langle V, E_i \rangle$,
\begin{eqnarray}\label{ht-c2-43}
\nonumber\sum_{i\geq 2} \frac{F^{11}h_{11i}^2}{h_{11}^2}&\leq& 2\sum_{i\geq 2} \frac{F^{11}\tau_i^2}{(\tau-a)^2}+2A^2\sum_{i\geq 2} F^{11}  \langle V, E_i \rangle^2\\
\nonumber&\leq& 2\frac{C_0}{a^2} \sum_{i\geq 2} \frac{aF^{11}h_{ii}^2}{\tau-a}+2n C_0A^2 F^{11}  \\
&\leq& \beta^2 \frac{2nC_0}{a^2}  \frac{aF^{11}h_{11}^2}{\tau-a}+
\frac{a}{16 (\tau-a)} F^{11}h_{11}^2.
\end{eqnarray}
By Cauchy-Schwarz inequality and Codazzi formula, we have
\begin{eqnarray}\label{CSC}
  -\frac{2}{h_{11}} \sum_{i\geq2}G^{1i,i1}h_{1i1}^2&=& -\frac{2}{h_{11}} \sum_{i\geq2}G^{1i,i1}(h_{11i}
  +\overline{R}_{01i1})^2  \\
 \nonumber &\geq& -\frac{2}{h_{11}} \sum_{i\geq2}G^{1i,i1}\left(\frac{3}{4}h_{11i}^2-3\overline{R}_{01i1}^2\right).
\end{eqnarray}
When we choose $\beta$ sufficiently small such that $ \beta\leq \min \left\{\sqrt{\frac{a^2}{32n C_0}}, \frac{1}{2}\right\}$, by \eqref{ht-c2-42},  \eqref{ht-c2-43}, \eqref{CSC} and Proposition \ref{th-lem-07}, we have
\begin{eqnarray}\label{ht-c2-44}
\nonumber\sum_{i\geq 2} \frac{F^{ii}h_{11i}^2}{h_{11}^2}
&\leq&-\frac{3}{2h_{11}} \sum_{i\geq 2}G^{1i, i1} h_{11i}^2 +\frac{a F^{11}h_{11}^2}{8 (\tau-a)}\\
\nonumber&\leq&-\frac{2}{h_{11}} \sum_{i\geq 2}G^{1i, i1} h_{1i1}^2+\frac{6}{h_{11}}\sum_{i\geq2}-G^{1i,i1}\overline{R}_{01i1}^2
 +\frac{a F^{11}h_{11}^2}{8 (\tau-a)}\\
\nonumber &\leq&-\frac{2}{h_{11}} \sum_{i\geq 2}G^{1i, i1} h_{1i1}^2+6C_0\sum_{i\geq2}\frac{G^{11}-G^{ii}}{h_{11}-h_{ii}}+\frac{a F^{11}h_{11}^2}{8 (\tau-a)}\\
 &\leq&-\frac{2}{h_{11}} \sum_{i\geq 2}G^{1i, i1} h_{1i1}^2+\frac{6C_0}{1-\beta}\sum_{i\geq2}(G^{11}-G^{ii})+\frac{a F^{11}h_{11}^2}{8 (\tau-a)},
\end{eqnarray}
if $h_{11}\geq B>1$. The last inequality comes from $\frac{1-\beta}{h_{11}-h_{ii}}\leq\frac{1}{h_{11}}<1$.
Note that
$$\sum_{i\geq2}(G^{11}-G^{ii})=nG^{11}-\sum_iG^{ii}\leq(n-1)\sum_iG^{ii}=\sum_iF^{ii}.$$
Then \eqref{ht-c2-44} gives that
\begin{eqnarray}\label{Gt}
\nonumber-\frac{2}{h_{11}} \sum_{i\geq 2}G^{1i, i1} h_{1i1}^2-\sum_{i\geq 2} \frac{F^{ii}h_{11i}^2}{h_{11}^2}
 &\geq&-\frac{6C_0}{1-\beta}\sum_{i\geq2}(G^{11}-G^{ii})-\frac{a F^{11}h_{11}^2}{8 (\tau-a)}\\
 &\geq&-12C_0\sum_iF^{ii}-\frac{a F^{11}h_{11}^2}{8 (\tau-a)}.
\end{eqnarray}
Substitute  \eqref{ht-c2-41}  and  \eqref{Gt} into \eqref{ht-c2-32}, then
\begin{eqnarray*}
\nonumber0&\geq& -\frac{F^{11}\tau_1^2}{(\tau-a)^2}-\frac{aF^{ii}h_{ii}^2}{8(\tau-a)}-12C_0\sum_iF^{ii}
-\frac{aF^{11}h_{11}^2}{8(\tau-a)}+\frac{aF^{ii}h_{ii}^2}{2(\tau-a)}\\
\nonumber&&+\frac{F^{ii}\tau_i^2}{(\tau-a)^2}+\frac{A\lambda'-C_0}{2}\sum_iF^{ii}-C_0(1+\sum_iF^{ii})-AC_0\\
\nonumber&\geq&\frac{a F^{ii}h_{ii}^2}{4(\tau-a)}+ \frac{A\lambda'-27C_0}{2} \sum_i F^{ii} -C_0(A+1)\\
&\geq&\frac{C_0}{2}h_{11}-C_0(A+1),
\end{eqnarray*}
which implies that
$$h_{11}\leq 2(A+1).$$
\end{proof}

%%%%%%%%%%%%%%%%%%%%%%%%%%%%%%%%%%%%%%%%%%%%%%%%%%%%%%%%%%%%%%%%%%%%%%%%%%%%%%
%%%%%%%%%%%%%%%%%%%%%%%%%%%%%%%%%%%%%%%%%%%%%%%%%%%%%%%%%%%%%%%%%%%%%%%%%%%%%%
\section{The proof of  Theorem  \ref{Main}}
%%%%%%%%%%%%%%%%%%%%%%%%%%%%%%%%%%%%%%%%%%%%%%%%%%%%%%%%%%%%%%%%%%%%%%%%%%%%%%
%%%%%%%%%%%%%%%%%%%%%%%%%%%%%%%%%%%%%%%%%%%%%%%%%%%%%%%%%%%%%%%%%%%%%%%%%%%%%%
In this section, we use the degree theory for nonlinear elliptic
equation developed in \cite{Li89} to prove Theorem \ref{Main}. The
proof here is similar to those in \cite{Al, Jin, Li-Sh}. So, only sketch will
be given below.

Based on a  priori estimates in Theorem \ref{n-2-C^0},
Theorem \ref{n-2-C1e} and Theorem \ref{n-2-C2e}, we know that the
equation \eqref{Eq2} is uniformly elliptic. From \cite{Eva82},
\cite{Kry83} and Schauder estimates, we have
\begin{eqnarray}\label{C2+}
|r|_{C^{4,\alpha}(M)}\leq C
\end{eqnarray}
for any $(\eta,k)$-convex solution $\Sigma$ to the equation \eqref{Eq2}, where
the position vector of $\Sigma$ is $V=(r(u), u)$ for $u \in M$.
We define
\begin{eqnarray*}
C_{0}^{4,\alpha}(M)=\{r \in
C^{4,\alpha}(M): \Sigma \ \mbox{is}
 \ (\eta,k)-\mbox{convex}\}.
\end{eqnarray*}
Let us consider $$F(.; t): C_{0}^{4,\alpha}(M)\rightarrow
C^{2,\alpha}(M),$$ which is defined by
\begin{eqnarray*}
F(r, u; t)=\frac{\sigma_k(\mu(\eta))}{\sigma_{l}(\mu(\eta))} -tf(r,u, \nu)-(1-t)\phi(r) \frac{C_n^k}{C_n^l}((n-1)\zeta(r))^{k-l}.
\end{eqnarray*}
Let $$\mathcal{O}_R=\{r \in C_{0}^{4,\alpha}(M):
|r|_{C^{4,\alpha}(M)}<R\},$$ which clearly is an open
set of $C_{0}^{4,\alpha}(M)$. Moreover, if $R$ is
sufficiently large, $F(r, u; t)=0$ has no solution on $\partial
\mathcal{O}_R$ by a priori estimate established in \eqref{C2+}.
Therefore the degree $\deg(F(.; t), \mathcal{O}_R, 0)$ is
well-defined for $0\leq t\leq 1$. Using the homotopic invariance of
the degree, we have
\begin{eqnarray*}
\deg(F(.; 1), \mathcal{O}_R, 0)=\deg(F(.; 0), \mathcal{O}_R, 0).
\end{eqnarray*}
Proposition \ref{Uni} shows that $r_0=1$ is the unique
solution to the above equation for $t=0$. Direct calculation shows
that
\begin{eqnarray*}
F(sr_0; 0)= (1-\phi(sr_0)) \frac{C_n^k}{C_n^l}((n-1)\zeta(r))^{k-l}(sr_0).
\end{eqnarray*}
Then
\begin{eqnarray*}
\delta_{r_0}F(r_0, u; 0)=\frac{d}{d s}|_{s=1}F(sr_0, u;
0)=-\phi'(r_0) \frac{C_n^k}{C_n^l}((n-1)\zeta(r))^{k-l}(r_0)>0,
\end{eqnarray*}
where $\delta F(r_0, u; 0)$ is the linearized operator of $F$ at
$r_0$. Clearly, $\delta F(r_0, u; 0)$ takes the form
\begin{eqnarray*}
\delta_{w}F(r_0, u; 0)=-a^{ij}w_{ij}+b^i
w_i-\phi'(r_0) \frac{C_n^k}{C_n^l}((n-1)\zeta(r))^{k-l}(r_0),
\end{eqnarray*}
where $a^{ij}$ is a positive definite matrix. Since
$-\phi'(r_0) \frac{C_n^k}{C_n^l}((n-1)\zeta(r))^{k-l}(r_0)>0,$
thus $\delta_{r_0} F(r_0, u; 0)$ is an invertible operator. Therefore,
\begin{eqnarray*}
\deg(F(.; 1), \mathcal{O}_R; 0)=\deg(F(.; 0), \mathcal{O}_R, 0)=\pm
1.
\end{eqnarray*}
So, we obtain a solution at $t=1$. This completes the proof of
Theorem \ref{Main}.

%%%%%%%%%%%%%%%%%%%%%%%%%%%%%%%%%%%%%%%%%%%%%%%%%%%%%%%%%%%%%%%%%%%%%%%%%%%%%%
%%%%%%%%%%%%%%%%%%%%%%%%%%%%%%%%%%%%%%%%%%%%%%%%%%%%%%%%%%%%%%%%%%%%%%%%%%%%%%


\bigskip

\bigskip

\begin{thebibliography}{99}

%\bibitem{Ale90} S. Alexander and R. Currier,
 %\emph{Nonnegatively curved hypersurfaces of hyperbolic space and subharmonic functions},
 %J. London Math. Soc.,  \textbf{41} (1990), 347-360.

%\bibitem{Ale93} S. Alexander and R. Currier, \emph{Hypersurfaces and nonnegative
%curvature}, Proc. Symp. Pure Math.,  \textbf{54} (1993), 37-44.

\bibitem{Al}
F.  Andrade, J. Barbosa and J.  de Lira,
\emph{Closed Weingarten hypersurfaces in warped product manifolds},
Indiana Univ. Math. J., \textbf{58} (2009), 1691-1718.



\bibitem{Ba-Li}
J.  Barbosa, J.  de Lira and V.  Oliker,
\emph{A priori estimates for starshaped compact hypersurfaces with prescribed mth curvature function in space forms},
 Nonlinear problems in mathematical physics and related topics I, Int. Math. Ser.(N. Y.), \textbf{1} (2002), 35-52.


\bibitem{Ca1}
L. Caffarelli, L. Nirenberg and J. Spruck,
\emph{Dirichlet problem for nonlinear second order elliptic equations. I. Monge-Amp\`ere equation},
Comm. Pure Appl. Math., \textbf{37} (1984), 369-402.




\bibitem{Ca}
L. Caffarelli, L. Nirenberg and J. Spruck,
\emph{Nonlinear second order elliptic equations, IV. Starshaped compact Weingarten hypersurfaces},
 Current Topics in PDEs, (1986), 1-26.

\bibitem{CQ1} C. Chen,
\emph{On the elementary symmetric functions}, Preprint.




 \bibitem{CJ}
J. Chu and H. Jiao,
\emph{Curvature estimates for a class of Hessian type equations},
	arXiv:2004.05463.

\bibitem{CLW}
 D. Chen, H. Li and Z. Wang,
 \emph{Starshaped compact hypersurfaces with prescirbed Weingarten curvature in warped product manifolds},
 arXiv:1705.00313.


\bibitem{CTX}
X.  Chen, Q.  Tu and N.  Xiang,
\emph{A class of Hessian quotient equations in Euclidean space},
J. Diff. Equa., \textbf{269} (2020), 11172-11194.

\bibitem{Eva82}
L.  Evans,
\emph{Classical solutions of fully nonlinear, convex, second-order elliptic equations},
Comm. Pure Appl. Math., \textbf{35} (1982), 333-363.




%\bibitem{Ga}
%P. Gauduchon,
%\emph{La $1$-forme de torsion d'une vari$\acute{e}$t$\acute{e}$ hermitienne compacte}, Math. Ann.,  \textbf{267} (1984), 495-518.






\bibitem{Guan02}
B. Guan and P. Guan,
 \emph{Convex hypersurfaces of prescribed curvatures},
 Ann. of Math.,  \textbf{156} (2002),  655-673.

\bibitem{Guan15}
P. Guan and J. Li,
\emph{A mean curvature type flow in space forms},
 Int. Math. Res. Not,     \textbf{13} (2015), 4716-4740.

\bibitem{Guan12}
P. Guan, J. Li and Y.  Li,
\emph{Hypersurfaces of Prescribed Curvature Measure},
Duke Math. J., \textbf{161} (2012), 1927-1942.

\bibitem{Guan09}
P. Guan, C. Lin and X. Ma,
\emph{The Existence of Convex Body with Prescribed Curvature Measures},
Int. Math. Res. Not., (2009), 1947-1975.

\bibitem{Guan-Ren15}
P. Guan, C. Ren and Z. Wang,
\emph{Global $C^2$ estimates for convex solutions of curvature equations},
 Comm. Pure Appl. Math.,  \textbf{68} (2015), 1287-1325.




\bibitem{HLW} Y. Hu, H. Li and Y. Wei,
\emph{Locally constrained curvature flows and geometric inequalities
in hyperbolic space},  arXiv:2002.10643.



\bibitem{HL2}
F. Harvey, H. Lawson,
\emph{p-convexity, p-plurisubharmonicity and the Levi
problem}, Indiana Univ. Math. J.,  \textbf{62} (2013), 149-169.

\bibitem{Iv1}
N. Ivochkina,
\emph{Solution of the Dirichlet problem for curvature equations of order $m$},
 Mathematics of the USSR- Sbornik,  \textbf{67} (1990), 317-339.

\bibitem{Iv2}
N. Ivochkina,
\emph{The Dirichlet problem for the equations of curvature of order $m$},
Leningrad Math. J.,  \textbf{2} (1991), 631-654.


\bibitem{Jin}
Q. Jin and Y.  Li,
\emph{Starshaped compact hypersurfaces with prescribed $k$-th mean curvature in hyperbolic space},
Discrete Contin. Dyn. Syst., \textbf{15} (2006),  367-377.


\bibitem{Kry83}
N.  Krylov,
\emph{Boundedly inhomogeneous elliptic and parabolic equations in a domain},
Izv. Akad. Nauk SSSR Ser. Mat.,
\textbf{47} (1983), 75-108.



\bibitem{Li-Sh}
Q.  Li and W.  Sheng,
\emph{Closed hypersurfaces with prescribed Weingarten curvature in Riemannian manifolds},
 Calc. Var. Partial Differential Equations, \textbf{48} (2013),  41-66.



\bibitem{Li89}
Y.  Li,
\emph{Degree theory for second order nonlinear elliptic operators and its applications},
Comm. Partial Differential Equations, \textbf{14} (1989), 1541-1578.

\bibitem{Li-Ol}
Y.  Li and V.  Oliker,
\emph{Starshaped compact hypersurfaces with prescribed $m$-th mean curvature in elliptic space},
J. Partial Differential Equations, \textbf{15} (2002),  68-80.



%\bibitem{LT94}
%M. Lin and N. Trudinger,
%\emph{On some inequalities for elementary symmetric functions},
%Bull. Aust. Math. Soc.,  \textbf{50} (1994), 317-326.


%\bibitem{Rei}
%R. Reilly,
%\emph{Variational properties of functions of the mean curvatures for hypersurfaces in space forms},
%J. Differential Geom.,  \textbf{8} (1973), 465-477.


\bibitem{Ren}
C. Ren and Z. Wang,
\emph{On the curvature estimates for Hessian equations},
 Amer. J. Math., \textbf{141} (2019), 1281-1315.

\bibitem{Ren1}
 C. Ren and Z. Wang,
 \emph{The global curvature estimate for the $n-2$ Hessian equation},
preprint, arXiv: 2002.08702.

%\bibitem{Ren2} C. Ren and Z. Wang,
%\emph{Notes on the curvature estimates for Hessian equations},
% arXiv:2003.14234.


\bibitem{Sha1}
J.  Sha,
\emph{p-convex Riemannian manifolds}, Invent. Math.,  \textbf{83} (1986), 437-447.

\bibitem{Sha2}
J.  Sha,
\emph{Handlebodies and p-convexity}, J. Differential Geom.,  \textbf{25} (1987),  353-361.




\bibitem{Sp}
J. Spruck and L. Xiao,
\emph{A note on starshaped compact hypersurfaces with prescribed scalar curvature in space form},
Rev. Mat. Iberoam. \textbf{33} (2017), 547-554.






%\bibitem{Tr90}
%N.  Trudinger,
%\emph{The Dirichlet problem for the prescribed curvature equations},
%Arch. Rational Mech. Anal.,  \textbf{111} (1990), 153-179.


\bibitem{Wu}
H. Wu,
\emph{Manifolds of partially positive curvature}, Indiana Univ. Math. J., \textbf{36} (1987), 525-548.



\end{thebibliography}
\end{document}